\documentclass[a4paper, reqno]{amsart}

\usepackage{amsthm}
\usepackage{amsfonts}
\usepackage{amsmath}
\usepackage{amssymb}
\usepackage{mathrsfs}
\usepackage[colorlinks=true,citecolor=black,linkcolor=black]{hyperref}
\usepackage{graphicx}
\usepackage{bbm}
\usepackage{enumitem}

\newtheorem{thm}{Theorem}
\newtheorem{thmm}{Theorem}
\newtheorem{lem}[thm]{Lemma}

\newtheorem{prop}{Proposition}

\theoremstyle{definition}

\theoremstyle{remark}

\DeclareMathAlphabet{\mathsc}{OT1}{cmr}{m}{sc}

\newcommand\cC{{\mathscr C}}

\newcommand\cK{{\mathcal K}}
\newcommand\cL{{\mathscr L}}

\newcommand\cZ{{\operatorname Z}}

\newcommand\bC{{\mathbb C}}

\newcommand\bN{{\mathbb N}}
\newcommand\bP{{\mathbb P}}

\newcommand\bR{{\mathbb R}}
\newcommand\bT{{\mathbb T}}
\newcommand\bZ{{\mathbb Z}}

\newcommand{\one}{\mathbf{1}}

\renewcommand{\L}[2]{\mathbf{L^{#1}}(#2)}
\newcommand{\norm}[1]{\left\lVert{#1}\right\rVert}
\newcommand{\abs}[1]{\left\lvert{#1}\right\rvert}

\newcommand{\ctaup}{C_{1}}
\newcommand{\cgam}{C_{\gamma}}
\newcommand{\ceta}{C_{2}}
\newcommand{\cly}{C_{\lambda}}
\newcommand{\cgrowth}{C_{\beta}}
\newcommand{\cm}{C_{3}}
\newcommand{\cother}{C_{4}}
\newcommand{\csurminus}{C_{5}}
\newcommand{\cjp}{C_{6}}
\newcommand{\ctauq}{C_{7}}
\newcommand{\cend}{C_{8}}

\newcommand{\ctransrate}{\gamma}
\newcommand{\climsup}{\gamma_{1}}
\newcommand{\cten}{\gamma_{2}}
\newcommand{\cyes}{\gamma_{3}}
\newcommand{\cseven}{\gamma_{4}}
\newcommand{\ceight}{\gamma_{5}}
\newcommand{\csurplus}{\gamma_{6}}
\newcommand{\cyonem}{\gamma_{7}}
\newcommand{\calmost}{\gamma_{8}}

\newcommand{\cone}{\rho_{1}}
\newcommand{\ctwo}{\rho_{2}}

\newcommand{\cfour}{\rho}

\newcommand{\nn}{n}

\newcommand{\Hhh}{H}
\newcommand{\measure}{\mathbf{m}}
\newcommand{\base}{\bT^{1}}

\numberwithin{equation}{section}

\begin{document}
\author{Oliver Butterley}
\address{Oliver Butterley\\ Fakult\"at f\"ur Mathematik\\ 
Universit\"at Wien\\
 Oskar-Morgenstern-Platz 1, 1090 Wien, Austria}
\email{oliver.butterley@univie.ac.at}
 
 \author{Peyman Eslami}
\address{Peyman Eslami\\ Dipartimento di Matematica\\
II Universit\`{a} di Roma (Tor Vergata)\\
Via della Ricerca Scientifica, 00133 Roma, Italy}
\email{eslami@mat.uniroma2.it}

\title{Exponential Mixing for Skew Products with Discontinuities}
\thanks{O.B. was supported by the Austrian Science Fund, Lise Meitner position M1583.
P.E. was supported by an INdAM-COFUND Marie Curie fellowship.
Both authors are grateful to the hospitality of Carlangelo~Liverani and ERC Advanced Grant MALADY (246953).}
\bibliographystyle{abbrv}
\begin{abstract}
We consider the skew product   $F: (x,u) \mapsto (f(x), u+\tau(x))$, 
where the base map $f : \bT^{1} \to \bT^{1}$ is piecewise $\cC^{2}$, covering and uniformly expanding, 
and the fibre map $\tau : \bT^{1} \to \bR$ is piecewise $\cC^{2}$.
We show the dichotomy that either this system mixes exponentially or
$\tau$ is cohomologous (via a Lipschitz function) to a piecewise constant.
\end{abstract}
\maketitle
\thispagestyle{empty}

\section{Introduction and Results}
In the study of dynamical systems, establishing the rate of mixing of a given system is of foremost importance.
It is a fundamental property describing the rate at which information about the system is lost. 
More importantly the rate of mixing (or typically slightly stronger information which is obtained whilst proving the rate of mixing) can be used to prove many other statistical properties (see, for example \cite[\S9]{Keller:1989} and \cite[Chapter~7]{MR2229799}).
Furthermore, of  physical relevance, these strong results associated to good rates of mixing are crucially used when studying weakly coupled systems~\cite{MR2530174,MR2842975}.

Rate of mixing results were first obtained for expanding maps and for hyperbolic maps (see \cite{Li1995} and references within),
then also for slower mixing, non-uniformly hyperbolic systems (e.g.,~\cite{Young1998, Young1999, MR1946554}).
In the case of hyperbolic flows or skew products like the one studied here one direction is completely neutral, with no expansion or contraction. These systems are not hyperbolic but merely partially hyperbolic.
 In these situations there is a  mechanism at work, different to hyperbolicity, but which  is nonetheless sufficient for producing good statistical properties including exponential rate of mixing. 
Dolgopyat~\cite{D}, extending work of Chernov~\cite{C}, succeeded in developing technology for studying this neutral mechanism and consequently proved exponential mixing for mixing Anosov flows when the stable and unstable invariant foliations are both $\cC^{1}$. 
Using and developing these ideas various results followed~\cite{Li1, baladi2005edc, avila2005emt,Ts}. 
However all the above systems were rather smooth or at least Markov.
Our knowledge concerning this same neutral mechanism in systems with discontinuities is less than satisfactory at present.
As far as the authors are aware, results of exponential mixing for hyperbolic flows with discontinuities are limited at present to the work of Baladi and Liverani~\cite{Baladi:2011uq}  for piecewise smooth 3D hyperbolic flows which preserve a contact structure and the work of Obayashi~\cite{Obayashi:2009} in the case of suspension semiflows over expanding maps with discontinuities which admit a Young tower. 

\begin{samepage}
In this article we study  the 2D skew product map $F:\bT^{2} \to \bT^{2}$ defined by 
\[
F: (x,u) \mapsto (f(x), u+\tau(x)).
\] \nobreak
  The base map $f : \bT^{1} \to \bT^{1}$ is required to be $\cC^{2}$ except for a finite number of discontinuities and admit a $\cC^{2}$ extension to the closure of the intervals of smoothness. 
  \end{samepage}
  Also $f$ is required to be  uniformly expanding and covering.\footnote{Covering implies that the unique absolutely continuous  invariant probability density is bounded away from zero~\cite{Liverani:1995aa}.}
The fibre map $\tau : \bT^{1} \to \bR$ is similarly required to be $\cC^{2}$ except for a finite number of discontinuities and admit a $\cC^{2}$ extension to the closure of the intervals of smoothness.\footnote{Here and throughout the document, if $u\in \base$, $s\in \bR$ then we consider $u+s\in \base$ in the natural sense that $\base = \bR \diagup \bZ$.}
At no stage do we require the map to be Markov, nor do we work with tower constructions to reduce to the Markov case.
Since the map $f$ is piecewise $\cC^{2}$ and uniformly expanding it is  known that there exists $\nu$ an  $f$-invariant probability measure which is absolutely continuous with respect to Lebesgue. 
Since the dynamics in the fibres is nothing more than a rigid rotation this means that  $\mu := \nu \times \operatorname{Leb}$ is an $F$-invariant probability measure on $\bT^{2}$.
Given observables $g,h : \bT^{2} \to \bC$ the correlation is defined as usual 
$\operatorname{Cor}_{g,h}(n) 
:=  \mu( g \cdot h \circ F^{n}) 
- \mu(g) \cdot \mu(h)$.
We say that $F:\bT^{2} \to \bT^{2}$ mixes exponentially if there exists $\zeta>0$ and for each pair of  H\"older continuous observables $g,h$  there exists  $C_{g,h}>0$ such that
$ \abs{ \operatorname{Cor}_{g,h}(n)   }
\leq 
C_{g,h}e^{-n\zeta}$, for all $n\in \bN$.

Our main result is the following.
\begin{thmm}
\label{thm:main}
 Let $F: \bT^{2}\to \bT^{2}$ be a piecewise-$\cC^{2}$ skew product over an expanding map as described above. 
 Either $F$ mixes exponentially 
 or 
 there exists Lipschitz $\theta : \base \to \bR$ and piecewise constant $\chi: \base \to \bR$ such that
 $\tau = \theta \circ f - \theta + \chi$. The discontinuites of $\chi$  only occur at points where either $f$ or $\tau$ are discontinuous.
\end{thmm}
\noindent
The remainder of this document is devoted to the proof of the above theorem.
The basic idea is from Dolgopyat~\cite{D}.
However we combine the best technology from the susequent articles~\cite{Li1,baladi2005edc,avila2005emt,Ts} in order to deal with the present difficulties, in particular the problems arising from  the discontinuities.

We note that the issue of the discontinuities could in theory be approached by using a tower construction and so reducing to the case of a base map which is Markov. This has been done by Obayashi~\cite{Obayashi:2009} in the case of suspension semiflows over expanding maps with discontinuities. However one particular problem with such tower constructions is that the tower is very sensitive to changes in the underlying system and so  important questions, for example determining the behaviour of statistical properties under perturbation of the original system, become completely unapproachable.

From a technical point of view we are forced in two opposing directions. To deal with discontinuities we are forced to consider densities of rather low regularity.  However we also need to take advantage of Dolgopyat's oscillatory cancellation argument which requires some good degree of regularity for the density.

The result for skew products is closely related to the analogous result for suspension semiflows. At a techincal level this can be seen from the twisted transfer operator (introduced below~\eqref{eq:twisttrans}) which is the same object used when studied in the context of skew products or flows (see, for example, \cite{oli1203}), with exactly the same estimates being required.

In section Section~\ref{sec:trans} the key notion of  transversality is discussed and a certain  estimate is shown to hold in the case when $\tau$ is not cohomologous to a piecewise constant. Section~\ref{sec:preparing} concerns the estimate of the norm of twisted transfer operators reducing the problem to a single key estimate (Proposition~\ref{prop:mainlem}). This key estimate is proven in Section~\ref{sec:mainestimate}, crucially using the transversality estimate from Section~\ref{sec:trans}. Finally, in Section~\ref{sec:rateofmixing}, the estimate on the twisted transfer operators is used to produce an estimate of exponential mixing.

\section{Transversality}
\label{sec:trans}
From this point onwards we will assume that $\tilde\lambda: = \inf f' > 2$. In general it would suffice to assume that there exists $n\in \bN$ such that $\inf \, (f^{n})' >1$. In that case we would simply consider a sufficiently large  iterate $m$ such that $\inf \, (f^{m})' > 2$  and proceed as now.
  Let $\Lambda := \sup \abs{f'} \geq \tilde\lambda$.  
We may assume that $ \sup \abs{\tau'} >0$ since if this does not hold then $\tau$ is  actually equal to a piecewise constant function, and in particular cohomologous to a piecewise constant function.
The first step is to define a forward invariant unstable conefield.
Let 
\[
\ctaup := \frac{ 2 \sup \abs{\tau'}}{\tilde \lambda - 1 } >0.
\]
 Define the constant conefield with the cones $\cK = \{ \bigl( \begin{smallmatrix}  \alpha\\ \beta \end{smallmatrix}  \bigr): \abs{ \smash{\frac{\beta}{\alpha}}} \leq \ctaup\}$. This conefield is strictly invariant  under 
 \[
 DF({x}) = \bigl(\begin{smallmatrix}
f'(x) & 0\\ \tau'(x) & 1
\end{smallmatrix} \bigr).
\]    
To see the invariance note that
$DF(x): \bigl( \begin{smallmatrix}  \alpha\\ \beta \end{smallmatrix}  \bigr) \mapsto \bigl( \begin{smallmatrix}  \alpha'\\ \beta' \end{smallmatrix}  \bigr) $ where ${\smash{\frac{\beta'}{\alpha'}}} = ( {\tau'}(x) + {\smash{\frac{\beta}{\alpha}}} )/f'(x)$.      
Let $x_{1}, x_{2} \in \base$ be two preimages of some $y\in \bT^{1}$, i.e., $f^{n}(x_{1}) = f^{n}(x_{2}) =y$. We write $x_{1} \pitchfork x_{2}$ (meaning \emph{transversal}) if $DF^{n}_{x_{1}} \cK \cap DF^{n}_{x_{2}} \cK = \{0\}$. Note that this transversality depends on $n$ even though the dependence is suppressed in the notation.
For future convenience let $J_{n} := {\abs{(f^{n})'}}^{-1}$.
Define the quantity
\[
\varphi(n)
:=
\sup_{y \in \base} \  \ \sup_{x_{1}\in f^{-n}(y)}  \sum_{ \substack{x_{2} \in f^{-n}(y) \\ x_{1} \not\pitchfork x_{2}} } J_{n}(x_{2}).
\]
This crucial quantity gives control on the fraction of preimages which are not transversal. 
In  this section we prove the following which is an extension of Tsujii~{\cite[Theorem~1.4]{Ts}} to the present situation where discontinuities are permitted.
Much of the argument follows the reasoning of the above mentioned reference with some changes due to the more general setting.
\begin{prop}
\label{prop:transversality}
Let $F:(x,u) \mapsto (f(x), u+\tau(x))$ be a piecewise-$\cC^{2}$ skew product over an expanding base map as above. 
\textbf{Either:}
\begin{equation}
\label{eq:transversality}
\limsup_{n\to \infty} \varphi(n)^{\frac{1}{n}}
<1,
\end{equation}
\textbf{Or:}
There exists Lipschitz $\theta : \base \to \bR$ and piecewise constant $\chi: \base \to \bR$ such that
 $\tau = \theta \circ f - \theta + \chi$. Moreover the discontinuities of $\chi$ only occur at points where either $f$ or $\tau$ are discontinuous.
\end{prop}

Before proving the above, let us record a consequence of the transversality. 
Let $\tau_{n} := \sum_{j=0}^{n-1} \tau \circ f^{j}$. 

\begin{lem}\label{lem:transverse}
If $f^{n}(x_{1}) = f^{n}(x_{2})$ and $x_{1} \pitchfork x_{2}$ 
then
\[
\abs{({\tau_{n}'} \cdot {J_{n}})(x_{1}) -({\tau_{n}'} \cdot {J_{n}})(x_{2})}
 >  {\ctaup} ( J_{n}(x_{1})+ J_{n}(x_{2}) ).
\]
\end{lem}
\begin{proof}
Assume that $\frac{\tau_{n}'}{(f^{n})'}(x_{1}) \geq  \frac{\tau_{n}'}{(f^{n})'}(x_{2})$, the other case being identical. 
Note that 
\[
DF^{n}({x_{1}})\bigl(\begin{smallmatrix}1\\ -\ctaup\end{smallmatrix}\bigr) = \left(\begin{smallmatrix}(f^{n})'(x_{1})\\ \tau_{n}'(x_{1}) -\ctaup\end{smallmatrix}\right),
\quad 
DF^{n}({x_{2}})\bigl(\begin{smallmatrix}1\\ \ctaup\end{smallmatrix}\bigr) = \left(\begin{smallmatrix}(f^{n})'(x_{2})\\ \tau_{n}'(x_{2}) + \ctaup\end{smallmatrix}\right).
\]
Transversality implies that 
$({\tau_{n}'(x_{1}) -\ctaup })/{( f^{n})'(x_{1})  }  > ({\tau_{n}'(x_{2}) + \ctaup}) / {( f^{n})'(x_{2})  }$.
\end{proof}

The remainder of this section is devoted to the proof of Proposition~\ref{prop:transversality}.
As mentioned in the introduction it is known that there exists an $f$-invariant probability measure $\nu$ which is absolutely continuous with respect to Lebesgue. 
Let $h_{\nu}$ denote the density of $\nu$. 
It is convenient to introduce the quantity
\begin{equation}
\label{eq:defphitilde}
\tilde\varphi(n,L,y)
:=
 \sum_{ \substack{x \in f^{-n}(y) \\ DF^{n}(x)\cK \supset L} } J_{n}(x) \cdot \frac{h_{\nu}(x)}{h_{\nu}(y)}
\end{equation}
where $L \in \bR\bP^{1}$ (an element of real projective space, i.e., a line in $\bR^{2}$ which passes through the origin).
Let $\tilde\varphi(n) :=  \sup_{y} \sup_{L} \tilde\varphi(n,L,y)$.
The benefit of this definition is that $\tilde\varphi(n) $ is submultiplicative,
i.e., $\tilde\varphi(n+m) \leq \tilde\varphi(n) \tilde\varphi(m)$ for all $n,m\in \bN$;
and $\tilde\varphi(n) \leq 1$ for all $n\in \bN$.

\begin{lem}
\label{lem:allequivalent}
The following statements are equivalent.
\begin{enumerate}[label={(\roman*)}, font=\normalfont]
\item
$\displaystyle\limsup_{n\to\infty} \varphi(n)^{\frac{1}{n}} =1$
\item
$\displaystyle\lim_{n\to\infty} \tilde\varphi(n)^{\frac{1}{n}} =1$,
\item
For all $n\in \bN$, $y\in \base$ there exists $L_{n}(y) \in \bR\bP^{1}$ such that, 
for every  $x \in f^{-n}(y)$,
$DF^{n}(x) \cK \supset L_{n}(y)$,
\item
There exists   a measurable $F$-invariant unstable direction, 
i.e., there exists $\ell: \base \to \bR$ such that  
$\tau' = f' \cdot \ell \circ f - \ell$
and so
\[
DF(x) \left(\begin{smallmatrix}  1 \\ \ell(x)  \end{smallmatrix}\right)
= f'(x) \left(\begin{smallmatrix}  1 \\ \ell\circ f(x)  \end{smallmatrix}\right).
\]
\item
Statement \textnormal{(iv)} holds with $\ell$  of bounded variation.
\item
There exists $ \theta : \base \to \base$ such that $\tau  - \theta\circ f + \theta$ is piecewise constant (discontinuities only where either $f$ or $\tau$ are discontinuous). Moreover $\theta$ is  differentiable with derivative of bounded variation.
\end{enumerate}
\end{lem}
Since $\limsup_{n\to\infty} \varphi(n)^{\frac{1}{n}} \leq 1$
the above lemma immediately implies Proposition~\ref{prop:transversality}. 
In the remainder of this section we prove the above lemma.
First a simple fact that we will use repeatedly.

\begin{lem}\label{lem:tauprime}
$\abs{J_{n} \cdot \tau_{n}'} \leq \frac{1}{2} \ctaup$.
\end{lem}
\begin{proof}
First observe that 
$\tau_{n}' = \sum_{i=0}^{n-1} \tau' \circ f^{i} \cdot (f^{i})'$.
 Consequently $\abs{J_{n} \cdot \tau_{n}'} \leq \abs{\tau'} \sum_{i=0}^{n-1} \tilde\lambda^{-i}$.
For all $n\in \bN$ the sum $\sum_{i=0}^{n-1} \tilde\lambda^{-i}$ is bounded from above by $(\tilde\lambda - 1)^{-1}$.
And so, using also  the definition of $\ctaup$,  
we know that $\abs{J_{n} \cdot \tau_{n}'} \leq  {\sup  \abs{\tau'} }/ ({\tilde\lambda - 1})   =  \frac{1}{2} \ctaup$.
\end{proof}

\begin{proof}[\bfseries Proof of  (i) $\Longrightarrow$ (ii)]
Suppose that $m \in \bN$, $n = n(m) = \lceil 2 \frac{  \ln \Lambda }{ \ln \tilde\lambda} m  \rceil$, 
$y\in \base$ and $x_{1}, x_{2} \in f^{-n}(y)$. 
Note that $n > m$ since $\Lambda \geq \tilde\lambda$.
Let $p=n-m$.
Further suppose that 
\[
DF^{n}(x_{1})\cK \cap DF^{n}(x_{2})\cK \neq \{0\}.
\]
The slopes of the edges of $DF^{n}(x_{1})\cK$ are
$\frac{\tau_{n}'}{(f^{n})'}(x_{1}) \pm \ctaup J_{n}(x_{1})$.
Let 
\[
L(x_{1}):= DF^{n}(x_{1})( \bR \times \{0\}).
\]
 The slope of $L$ is $\frac{\tau_{n}'}{(f^{n})'}(x_{1})$.
Since we assume the cones $DF^{n}(x_{1})\cK$ and $ DF^{n}(x_{2})\cK$ are not transversal this implies that
the difference in slope between one of the edges of $DF^{n}(x_{2})\cK$ and $L$ is not greater than 
\begin{equation}
\label{eq:rain}
 \ctaup J_{n}(x_{1}) \leq \ctaup \tilde\lambda^{-n}.
 \end{equation}
 Now consider the  cone $DF^{n}(x_{2})\cK$ and the  cone $DF^{m}(f^{p}x_{2})\cK \supset DF^{n}(x_{2}) \cK$.
The slopes of the edges of the first are
\[
\frac{\tau_{n}'}{(f^{n})'}(x_{2}) \pm \ctaup J_{n}(x_{2}) 
=  \frac{\tau_{m}'}{(f^{m})'}\circ f^{p}(x_{2}) + \frac{\tau_{p}'}{(f^{m})'\circ f^{p} \cdot (f^{p})'}(x_{2})  \pm \ctaup J_{n}(x_{2}),
\]
whilst the slopes of the edges of the second are
\[
\frac{\tau_{m}'}{(f^{m})'}\circ f^{p}(x_{2}) \pm \ctaup J_{m}\circ f^{p}(x_{2}).
\]
Consequently the slopes of the edges of the two cones are separated by at least
\[
J_{m}\circ f^{p}(x_{2}) \left(  \ctaup   -  \sup \abs{\tau_{p}' \cdot J_{p}}    \right) - \ctaup J_{n}(x_{2}).
\]
By Lemma~\ref{lem:tauprime} we know that  $\abs{\tau_{p}' \cdot J_{p}}  \leq \frac{1}{2} \ctaup$.
This means that the above term is bounded from below by
\[
\tfrac{1}{2} \ctaup \Lambda^{-m}  -  \ctaup \tilde\lambda^{-n}
\geq \tfrac{1}{2} \ctaup  \tilde\lambda^{-\frac{n}{2}}  -  \ctaup \tilde\lambda^{-n},
\]
where we used that the assumed relation between $n$ and $m$ implies that $m \leq \frac{n}{2} \frac{\ln \lambda}{\ln \Lambda}$ and so $\Lambda^{-m} \geq \tilde\lambda^{-\frac{n}{2}}$.
Recall now~\eqref{eq:rain}. For all $n$ sufficiently large then
$ \tfrac{1}{2} \ctaup  \tilde\lambda^{-\frac{n}{2}}  -  \ctaup \tilde\lambda^{-n}
\geq \ctaup \tilde\lambda^{-n}$.
To conclude, we have shown that
$DF^{n}(x_{1})\cK \cap DF^{n}(x_{2})\cK \neq \{0\}$
implies that
$DF^{m}(f^{p}x_{2})\cK \supset L(x_{1})$ where $L(x_{1})$ is defined as before.
This means that 
\[
\sum_{ \substack{x_{2} \in f^{-n}(y) \\ x_{1} \not\pitchfork x_{2}} } J_{n}(x_{2})
\leq
 \sup_{L} \sum_{ \substack{x \in f^{-m}(y) \\ DF^{m}(x)\cK \supset L} } J_{m}(x) . 
\]
Finally this implies that $\varphi(n) \leq \ceta \tilde \varphi(m(n))$, where $\ceta:= {\sup h_{\nu}}/{\inf h_{\nu}} >0$.
\end{proof}

\begin{proof}[\bfseries Proof of  (ii) $\Longrightarrow$ (iii)]
By submultiplicativity  and the fact that $\tilde \varphi(n) \leq 1$ for all $n\in \bN$ the assumption $\lim_{n\to\infty} \tilde\varphi(n)^{\frac{1}{n}} =1$ 
implies that $ \tilde\varphi(n) =1$ for all $n\in \bN$.
Consequently the following statement holds:
\begin{quotation}
For all $n$ there exists $y_{n}\in \base$ and  $L_{n} \subset \bR\bP^{1}$
such that,  for all $x\in f^{-n}(y_{n})$, 
$DF^{n}(x)\cK \supset L_{n}$. 
\end{quotation}
It remains to prove that this above statement implies statement~(iii).
We will prove the contrapositive. Suppose the negation of statement~(iii). I.e., 
there exists $n_{0}\in \bN$, $y\in \base$, $x_{1}, x_{2} \in f^{-n_{0}}(y)$ such that $DF^{n_{0}}(x_{1}) \cK  \cap DF^{n_{0}}(x_{2}) \cK = \{0\} $.
Let  $g_{1}, g_{2}$ denote the two inverse maps corresponding to $x_{1},x_{2}$. These inverses are defined on some interval containing $y$ and due to the openness of the transversality of cones we can assume that 
$DF^{n_{0}}(g_{1}(y)) \cK  \cap DF^{n_{0}}(g_{2}(y)) \cK = \{0\} $ for all $y \in \omega_{*}$ where $\omega_{*} \subset \base$ is an open interval.
Since $f$ is covering   there exists $m_{0}\in \bN$ such that 
$ f^{m_{0}}(\omega_{*})
=\base$.
Let $m= m_{0}+n_{0}$.
For all $y\in \base$ there exists $z\in f^{-m_{0}}(y)$ and there exists  $x_{1}, x_{2} \in f^{-n_{0}}(z)$ with the above transversality property.
This means that for all $y\in \base$ there exist $x_{1}, x_{2} \in f^{-m}(y)$ such that
$DF^{m}(x_{1}) \cK  \cap DF^{m}(x_{2}) \cK = \{0\} $, since
\[
DF^{m}(x_{1}) \cK  \cap DF^{m}(x_{2}) \cK 
= DF^{m_{0}}(y) ( DF^{n_{0}}(x_{1}) \cK  \cap DF^{n_{0}}(x_{2}) \cK). 
\]
This contradicts the above statement concerning the existence of some $L_{n}$ such that $DF^{n}(x)\cK \supset L_{n}$ for all $x\in f^{-n}(y)$.
\end{proof}

\begin{proof}[\bfseries Proof of  (iii) $\Longrightarrow$ (iv)]
For all $x\in \base$ let $\ell_{n}(x) $ denote the slope of $L_{n}(x)$. 
I.e., $\left(\begin{smallmatrix}
       1 \\ \ell_{n}(x)
      \end{smallmatrix}\right)
      \in L_{n}(x)$.
The uniform expansion means that the image of unstable cones contracts and consequently for each $x$ then $\ell_{n}(x) \to \ell(x)$ as $n\to \infty$.     
The function $\ell(x)$
enjoys the property that $\tau_{n}'(x) + \ell(x) = (f^{n})'(x) \cdot \ell(f^{n}x)$.
\end{proof}

\begin{proof}[\bfseries Proof of  (iv) $\Longleftrightarrow$ (v)]
The implication (v) $\Longrightarrow$ (iv) is immediate.
Assume that statement (iv) holds.
Since $\ell$ is invariant
we know that for any $n\in \bN$, $x \in f^{-n}(y)$ that
\[
 \ell(y) = \frac{ \tau_{n}'}{(f^{n})'} (x)+  \frac{\ell}{(f^{n})'} (x).
\]
For large $n$ the second term on the right hand side becomes very small. 
Note that, because we assume (iv) holds, if we want to calculate $\ell$ at $y$ it does not matter which preimage $x$ we consider. 
Fix some $\omega_{0} \subset \base$ a disjoint union of intervals and a bijection $g: \omega_{0} \to \base $ such that $f \circ g$ is the identity. We can do this in such a way that $g$ is $\cC^{2}$ on each  component of $\omega_{0}$\footnote{Note that if the map $f$ was full branch we could choose $g$ to be $\cC^2$ but this cannot be expected in general.}.
Of course $f^{n} \circ g^{n} = \mathbf{id}$ for all $n\in \bN$.
Consequently
\[
\begin{aligned}
  \ell & = \frac{\tau_{n}'}{(f^{n})'}\circ g^{n}  +  \frac{\ell}{(f^{n})'}\circ g^{n} 
\\
& = 
 \sum_{j=0}^{n-1}
 \frac{\tau'}{(f^{n-j})'}\circ g^{n-j}  
 +  \frac{\ell}{(f^{n})'}\circ g^{n}.
\end{aligned}
\]
Note that 
$\norm{\smash{\frac{\ell}{(f^{n})'}\circ g^{n}}}_{\L{\infty}{\base}} \to 0$ 
as $n\to\infty$.
Also note that $ \sum_{j=0}^{\infty}
 \frac{\tau'}{(f^{j})'}\circ g^{j}$ is of bounded variation. Indeed each term in this infinite sum is piecewise $\cC^{2}$ and  has only a finite number of discontinuities. 
Moreover the $\mathbf{BV}$ norm of the terms is exponentially decreasing due to the uniform expansion and so the sum converges in  $\mathbf{BV}$.
Consequently $\ell$ must be of bounded variation.
\end{proof}

\begin{proof}[\bfseries Proof of  (v) $\Longleftrightarrow$ (vi)]
First we prove  (v) $\Longrightarrow$ (vi).
For all $y\in \base$ let
\[
 \theta(y) := \int_{0}^{y} \ell(x) \ dx.
\]
This defines a Lipschitz function on $\base$, differentiable in the sense that the derivative is of bounded variation.
There exists a partition ${\{\omega_{m}\}}_{m}$ such that $\tau$ and $f$ are $\cC^{2}$ when restricted to each element of the partition.
Write $\omega_{m} = (a_{m}, b_{m})$.
If $y\in \omega_{m}$ then
$\tau(y) = \tau(a_{j}) + \int_{a_{j}}^{y} \tau'(x) \ dx$.
Substituting the equation $\tau' = f' \cdot \ell \circ f - \ell$ we obtain
\[
 \begin{aligned}
 \tau(y) 
 &= \tau(a_{j}) 
 +  \int_{a_{j}}^{y}  f' \cdot \ell \circ f(x) \ dx - \int_{a_{j}}^{y}  \ell(x)  \ dx\\
 &= \tau(a_{j}) 
 +  \int_{f(a_{j})}^{f(y)}   \ell (x) \ dx - \int_{a_{j}}^{y}  \ell(x)  \ dx \\
 &= \theta \circ f(y) - \theta(y) 
 +\left( \theta(a_{j})  - \theta \circ f(a_{j})
+  \tau(a_{j})\right)\\
& = \theta\circ f(y) -  \theta(y) + \chi_{j}.
 \end{aligned}
\]
Let $\chi$ denote the piecewise constant function equal to $\chi_{j}$ on each $\omega_{j}$.
The implication (vi) $\Longrightarrow$ (v) follows by differentiating $\tau - \theta\circ f + \theta = \chi$.
\end{proof}

\begin{proof}[\bfseries Proof of  (iv) $\Longrightarrow$ (i)]
The vector $\left(\begin{smallmatrix}
                   1 \\ \ell(y)
                  \end{smallmatrix} \right)$
                  is contained within $DF^{n}(x)\cK$ for all $x\in f^{-n}(y)$ since 
                  $\left(\begin{smallmatrix}
                   1 \\ \ell(x)
                  \end{smallmatrix} \right) \in \cK$
   and $DF^{n}(x) \left(\begin{smallmatrix}
                   1 \\ \ell(x)
                  \end{smallmatrix} \right) 
                  = (f^{n})'(x) 
 \left(\begin{smallmatrix}
                   1 \\ \ell(y)
                  \end{smallmatrix} \right)$.
       Consequently            
    $x_{1} \not\pitchfork x_{2}$, i.e.,   $   DF^{n}(x_{1})\cK \cap DF^{n}(x_{2})\cK \neq \{0\}$, for every $x_{1}, x_{2} \in f^{-n}(y)$.
\end{proof}

\section{Preparation for the Main Estimate}
\label{sec:preparing}

Throughout this section and the next we assume that the first alternative of Proposition~\ref{prop:transversality} holds.
Let $\climsup := \limsup_{n\to \infty} \varphi(n)^{\frac{1}{n}}
<1$,
and fix  $\ctransrate \in (0, \climsup)$.
There exists  $\cgam>0$ such that 
\begin{equation}
 \label{eq:transdecay}
 \varphi(n) \leq \cgam e^{-n\ctransrate}
 \quad
 \text{for all $n\in \bN$}. 
\end{equation}
The twisted transfer operator, for all  $b\in \bR$, $n\in \bN$ is given by the formula
\begin{equation}
\label{eq:twisttrans}
\cL_{b}^{n}h(y) =  \sum_{x\in f^{-n}(y)} J_{n} (x) \cdot h(x)\cdot e^{ib \tau_{n}(x)}.
\end{equation}
A simple estimate shows that $ \norm{\cL_{b}^{n}  h}_{\L{1}{\base} }  \leq  \norm{h}_{\L{1}{\base} }$.
We will work extensively with functions of \emph{bounded variation} due to the suitability of this function space for discontinuities. The Banach space is denoted $(\mathbf{BV} ,\norm{\cdot}_{\mathbf{BV}})$, variation is denoted by $\operatorname{Var}(\cdot)$, and $\norm{\cdot}_{\mathbf{BV}} : = \operatorname{Var}(\cdot) +  \norm{\cdot}_{\L{1}{\base} }$ as usual.
We have the following \emph{Lasota-Yorke} inequality.
\begin{lem}\label{lem:LY}
There exists $\lambda>0$, $\cly>0$ such that, for all $n\in \bN$, $b\in \bR$, $h\in \mathbf{BV}$
\[
\norm{\cL_{b}^{n} h}_{\mathbf{BV}} \leq \cly \lambda^{-n} \norm{h}_{\mathbf{BV}} + \cly(1 + \abs{b}) \norm{h}_{\L{1}{\base} }
\]
\end{lem}
\begin{proof}
The proof is essentially standard (see for example~\cite{Keller:1989}) but it is important to note the factor of $\abs{b}$ which appears in front of the $\mathbf{L^{1}}$ norm. 

We already know that $ \norm{\cL_{b}h}_{\mathbf{L^{1}} } \leq  \norm{h}_{\mathbf{L^{1}} }$.
Note that 
\[
\operatorname{Var}(h) = \sup \left\{   \norm{h\cdot \eta' }_{\mathbf{L^{1}} } : \eta \in \cC^{1}(\base, \bC), \abs{\eta} \leq 1  \right\}.
\]
Consequently we must estimate $\norm{\cL_{b} h\cdot \eta' }_{ \mathbf{L^{1}} } = \norm{h\cdot (e^{ib\tau} \cdot \eta'\circ f) }_{ \mathbf{L^{1}} }$. In order to do this note that (for convenience we denote $J := J_{1} = 1/\abs{f'}$)
\[
\left[  J \cdot \eta \circ f \cdot e^{ib\tau}  \right]'
= J' \cdot \eta \circ f \cdot e^{ib\tau}
+ ib \tau' \cdot J \cdot \eta \circ f \cdot e^{ib\tau}
+  (e^{ib\tau} \cdot  \eta'\circ f).
\]
This means that 
\[
\operatorname{Var}(\cL_{b}h)
\leq
\norm{J'}_{\mathbf{L^{\infty}}}
\norm{h}_{\mathbf{L^{1}}}
+
\abs{b} \norm{\tau' \cdot J}_{\mathbf{L^{\infty}}}
\norm{h}_{\mathbf{L^{1}}}
+
\norm{h \cdot \smash{ \left[  J \cdot \eta \circ f \cdot e^{ib\tau}  \right]' } }_{\mathbf{L^{1}}}
\]
The remaining problem  is that $[  J \cdot \eta \circ f \cdot e^{ib\tau}  ]$ could be discontinuous. Therefore we introduce the quantity $\phi : \base \to \bR$ which is piecewise affine (discontinuous only where $[  J \cdot \eta \circ f \cdot e^{ib\tau}  ]$ is discontinuous) and such that 
$([  J \cdot \eta \circ f \cdot e^{ib\tau}  ] - \phi)(x)$ tends to $0$ as $x$ approaches any discontinuity point.
This means that $[  J \cdot \eta \circ f \cdot e^{ib\tau}  ] $  is continuous and  piecewise\footnote{That $[  J \cdot \eta \circ f \cdot e^{ib\tau}  ] $  is continuous and  piecewise~$\cC^{1}$ means that it may be approximated by a $\cC^{1}$ function with error small in the appropriate sense that makes no difference to the final estimate.}~$\cC^{1}$.
Note that $\norm{\phi}_{\mathbf{L^{\infty}}} \leq \norm{J}_{\mathbf{L^{\infty}}}$ and so
$\norm{[  J \cdot \eta \circ f \cdot e^{ib\tau}  ] - \phi}_{\mathbf{L^{\infty}}} \leq 2 \norm{J}_{\mathbf{L^{\infty}}}$. 
On the other hand, taking advantage of the finite number of discontinuities in this setting, we know that  $\norm{\phi'}_{\mathbf{L^{\infty}}} $ is bounded by some constant which depends on the size of the smallest image of an element of the partition of smoothness. 
We have shown that
\[
\operatorname{Var}(\cL_{b}h)
\leq
2  \norm{J}_{\mathbf{L^{\infty}}} \operatorname{Var}(h)
+
( \norm{J'}_{\mathbf{L^{\infty}}} + \abs{b} \norm{\tau' \cdot J}_{\mathbf{L^{\infty}}} + \norm{\phi'}_{\mathbf{L^{\infty}}} )
\norm{h}_{\mathbf{L^{1}}}.
\]
This suffices\footnote{By considering higher iterates of the same argument, if one were interested in optimal estimates, $\lambda$ can be chosen arbitrarily close to $ \limsup_{n\to \infty} \smash{  \abs{J_{n}}^{\frac{1}{n}} }$.}
 since we assumed that $\inf\abs{f'}>2$ and so $2  \norm{J}_{\mathbf{L^{\infty}}} <1$.
Consequently the above estimate may be iterated to produce an estimate for all $n\in \bN$. 
\end{proof}
These estimates and the compactness of the embedding $\mathbf{BV} \hookrightarrow  \L{1}{\base}$, by the usual arguments (see, for example~\cite{keller2005sgo}), imply that the operator $\cL_{b} : \mathbf{BV} \to \mathbf{BV}$ has spectral radius not greater than $1$ and essential spectral radius not greater than $\lambda^{-1} \in (0,1)$. 
The spectral radius of $\cL_{0}  : \mathbf{BV} \to \mathbf{BV}$ is equal to $1$.

It is convenient to introduce the equivalent norm
\[
\norm{h}_{(b)} := (1+\abs{b})^{-1}  \norm{h}_{\mathbf{BV}}  + \norm{h}_{\L{1}{\base} }.
\]
The main purpose of this section is to 
prove the following result.
\begin{prop}\label{prop:mainuse}
 There exists $b_0 >0$, $\cfour>0$, and $\cten>0$ such that
 \[
  \norm{\smash{\cL_{b}^{n(b)}} }_{(b)}
  \leq
  e^{-n(b) \cten },
  \quad \quad\text{for all $\abs{b} \geq b_0$,  $n(b) := \lceil \cfour \ln \abs{b}  \rceil$}.
 \]
  \end{prop}
\noindent
The remainder of the section will be devoted to the proof of the above. The proof is self contained apart from using Proposition~\ref{prop:mainlem} (see below) whose proof is postponed to Section~\ref{sec:mainestimate}.

\begin{lem}\label{lem:LY2}
For all $n\in \bN$, $h\in \mathbf{BV}$, $b\in \bR$
\[
\norm{\cL_{b}^{n} h}_{(b)} \leq \cly \lambda^{-n} \norm{h}_{(b)} + \cly \norm{h}_{\L{1}{\base} }.
\]
\end{lem}
\begin{proof}
This is a direct result of the definition of the norm and the Lasota-Yorke estimate (Lemma~\ref{lem:LY}).
\end{proof}

First we deal with the easy case when $ \norm{ h}_{\mathbf{BV}} $ is large in comparison to $\norm{h}_{\L{1}{\base} } $.
Let $n_{0}:= \lceil \ln (4 \cly) / \ln \lambda   \rceil$.
\begin{lem}\label{lem:easy}
Suppose that $h\in \mathbf{BV}$, satisfying 
 $2 \cly (1+\abs{b}) \norm{h}_{\L{1}{\base} } \leq   \norm{h}_{\mathbf{BV}}$.
 Then $\norm{\cL_{b}^{n_{0}} h}_{(b)} \leq \frac{3}{4} \norm{ h}_{(b)} $.
 \end{lem}
\begin{proof}
The definition of $n_{0} \in \bN$ is such that $\cly \lambda^{-n_{0}} + \frac{1}{2} \leq   \frac{3}{4}$.
The conclusion then follows from Lemma~\ref{lem:LY2}.
\end{proof}

This means that we only need to worry about estimating in the case where  $2 \cly (1+\abs{b}) \norm{h}_{\L{1}{\base} } >   \norm{h}_{\mathbf{BV}}$.
This is the case where the density can be considered to be ``almost constant'' as long as we look on the scale of $\abs{b}^{-1}$.
Furthermore it will suffice to estimate the $\mathbf{L^{1}}$ norm and not the $\mathbf{BV}$ norm as demonstrated by the following calculation.
Using Lemma~\ref{lem:LY2}, for any $n\in \bN$
\begin{equation}\label{eq:iterate}
\begin{aligned}
\norm{\cL_{b}^{2n} h}_{(b)} 
&\leq \cly \lambda^{-n} \norm{\cL_{b}^{n} h}_{(b)} + \cly \norm{\cL_{b}^{n} h}_{\L{1}{\base} }\\
&\leq \cly^{2} \lambda^{-2n} \norm{ h}_{(b)}  + \cly^{2} \lambda^{-n} \norm{h}_{\L{1}{\base} }   + \cly \norm{\cL_{b}^{n} h}_{\L{1}{\base} }\\
&\leq  2 \cly^{2} \lambda^{-n} \norm{ h}_{(b)}   + \cly \norm{\cL_{b}^{n} h}_{\L{1}{\base} }.
\end{aligned}
\end{equation}
It therefore remains to obtain exponential contraction of  $ \norm{\cL_{b}^{n} h}_{\L{1}{\base} }$  in terms of $\norm{ h}_{(b)}$ in the case when $ 2 \cly (1+\abs{b} )\norm{h}_{\L{1}{\base} } > \norm{h}_{\mathbf{BV}}$.

In order to later deal with discontinuities we now introduce a  ``growth lemma'' suitable for this setting.
Fix $\delta>0$ such that, for any interval $\omega \subset \base$ of size $\abs{\omega}\leq \delta$  the image $f\omega$ consists of at most two connected components.
We will define unions of open intervals $\Omega_{n}$ for all $n\in \bN$ iteratively.
Let $\Omega_{0} \subset \base$ be an interval, $\abs{\Omega_{0}}\leq \delta$.
Suppose that $\Omega_{n}$ is already defined.
Let  $\omega$ be one of the connected components of $\Omega_{n}$. 
The image $f \omega$ is  the union of intervals, some could be large, some could be small. It is convenient to maintain all intervals of size less than $\delta$ and so we artificially chop long intervals so that they are always of size greater than $\delta/2 $ and less than $\delta$. 
In this fashion let the set ${\{\omega_{k}\}}_{k}$ be a set of open intervals which exhaust $\omega$ except for a zero measure set and such that each $f\omega_{k}$ is a single interval of length not greater than $\delta$.
The set  $\Omega_{n+1} $ is defined to be partition of $\Omega_{n}$ produced by following the same procedure for each connected component of $\Omega_{n}$.

We must control the measure of points close to the boundaries of $\Omega_{n} =  {\{ \omega_{j} \}}_{j}$.
For any $n\in \bN$, $x\in \Omega_{n}$, let $r_{n}(x) := d(f^{n}(x), f^{n}( \partial \Omega_{n}) )$,
and hence let ($\measure$ denotes Lebesgue measure)
\[
\cZ_{\epsilon}\Omega_{n} :=
\measure({ \{ x\in \Omega_{n}  : r_{n}(x) \leq \epsilon  \}  }).
\]
Let $\cgrowth:= {4 \Lambda \beta} {\delta^{-1} \lambda^{-1} (\beta -{1})^{-1}}$.
\begin{lem}
\label{lem:growth}
There exists $ \beta >1$ such that, for all $n\in \bN$, $\epsilon>0$,
\[
\cZ_{\epsilon}\Omega_{n} \leq
 \beta^{-n} \lambda^{n} \cZ_{\epsilon/ \lambda^{n}}\Omega_{0} 
+   \epsilon  \cgrowth \abs{ \Omega_{0}  }.
\]
\end{lem}
\begin{proof}
Suppose for the moment that $\Omega_{n}$ consists of just one element, i.e., $\Omega_{n} = \{\omega\}$.
We will estimate $\cZ_{\epsilon}\Omega_{n+1} $.
The image $f \omega$ consists of at most two connected components. 
But some of these connected components could be large in which case they will be cut into smaller pieces of size between $\delta/2 $ and  $\delta$.
The set $ \partial \Omega_{n+1}$ consists of points which come from one of  three different origins:
from $ \partial \Omega_{n}$; from a cut due to the discontinuities of the map; or from the artificial cuts. The first two possibilities are bounded by $2 \cZ_{\epsilon/ \lambda}\Omega_{n} $. The total length of $f \omega$ is not greater than $\Lambda \measure( \omega )$ and so the total number of artificial cuts is not greater than $2 \delta^{-1} \Lambda \measure( \omega )$. Summing these terms  we obtain the estimate
\[
\cZ_{\epsilon}\Omega_{n+1} \leq
2 \cZ_{\epsilon/ \lambda}\Omega_{n} 
+ 4  \epsilon \frac{ \Lambda } {\delta \lambda}\measure( \omega ).
\]
The equivalent estimate holds, even when $\Omega_{n}$ consists of more than  one element.
Choose $ \beta > 1$ such that $2 = \beta^{-1} \lambda$ and so the above estimate reads as 
\[
\cZ_{\epsilon}\Omega_{n+1} \leq
 \beta^{-1} \lambda \cZ_{\epsilon/ \lambda}\Omega_{n} 
+ 4  \epsilon \frac{ \Lambda } {\delta \lambda}\measure( \Omega_{n}  ),
\]
and iterated produces the estimate (since $\sum_{j=0}^{\infty}\beta^{-n} = \frac{\beta}{\beta-{1}}$)
\[
\cZ_{\epsilon}\Omega_{n} \leq
 \beta^{-n} \lambda^{n} \cZ_{\epsilon/ \lambda^{n}}\Omega_{0} 
+   \epsilon \frac{4 \Lambda \beta } {\delta \lambda (\beta -{1})}\measure( \Omega_{0}  ).\qedhere
\]
\end{proof}

The argument will depend crucially on the  three quantities $\cone, \xi, \ctwo>0$.
Let
\begin{equation}
 \label{eq:defconsts}
\cone := \frac{2}{\ln \lambda},
 \quad \quad
 \xi :=  \frac{\ln \beta}{2 \ln \lambda},
 \quad \quad
 \ctwo := \frac{\xi}{2\ln \Lambda},
\end{equation}
and hence let $n_{1}(b) := \left\lceil \cone \ln \abs{b} \right\rceil $, 
$n_{2}(b) := \left\lceil \ctwo \ln \abs{b} \right\rceil $.
Let $\nn(b) := n_{1}(b) + n_{2}(b)$.
For notational simplicity we will often suppress the dependence on $b$ of $\nn, n_{1}, n_{2}$.
Note that $\lambda \beta^{-1} =2$ and so $\ln \beta < \ln \lambda$ and hence $\xi < \frac{1}{2}$.
We use two time scales:
The first  $n_{1}$ iterates are for a small interval of length $\abs{b}^{-(1+\xi)}$ to expand to a reasonable size,
then we take $n_{2} $ iterates to see oscillatory cancelations. 
The argument will also depend on the choice of $b_{0}>0$. As several points during the argument this quantity will be chosen sufficiently large.

Denote by ${\{\Hhh_{\ell}\}}_{\ell}$ the partition of $\base$ into   subintervals 
of equal length
such that 
\begin{equation}
 \label{eq:Hpart}
\abs{b}^{-(1+\xi)} \leq \abs{ \Hhh_{\ell} } \leq 2\abs{b}^{-(1+\xi)}.
\end{equation}
We use this partition to approximate the density $h$. 
Denote by $h_{b}$ the density which is constant on each $\Hhh_{\ell}$ 
and equal to the average value of $h$ on $\Hhh_{\ell}$. 
Note that  $\norm{ h }_{\L{1}{\base}} = \norm{ h_{b} }_{\L{1}{\base}} $.

\begin{lem}\label{lem:approx}
Let  $h\in \mathbf{BV}$,  $b\in \bR$, $\abs{b} \geq b_{0}$ such that  $ 2 \cly (1+\abs{b} )\norm{h}_{\L{1}{\base} } > \norm{h}_{\mathbf{BV}}$ and let $ h_{b}$ be the piecewise constant function as defined in the above paragraph. 
Then
\[
\norm{\smash{h -  h_{b}}}_{\L{1}{\base} }  \leq 4 \cly e^{-\nn\frac{\xi}{2(\cone + \ctwo)}}  \norm{h}_{\L{1}{\base}}.
\]
\end{lem}
\begin{proof}
Standard approximation results  for $\mathbf{BV}$ and Sobolev functions imply that
$\norm{\smash{h -  h_{b}}}_{\L{1}{\base} }  \leq 2 \abs{b}^{-(1+\xi)} \norm{h}_{\mathbf{BV}}$
since $\abs{\Hhh_{\ell}} \leq 2 \abs{b}^{-(1+\xi)}$.
Substituting the control on $ \norm{h}_{\mathbf{BV}}$ which is assumed we have
\[
\norm{\smash{h -  h_{b}}}_{\L{1}{\base} }  \leq 4 \abs{b}^{-(1+\xi)} \cly (1+\abs{b} )\norm{h}_{\L{1}{\base} }. 
\]
Ensuring that $b_{0}>1$ then $(1+ \abs{b} )\leq 2 \abs{b}$. 
Increasing $b_{0}$ more if required we may assume that
$\nn(b) \leq 2(\cone+\ctwo) \ln \abs{b}$.
This means that $\abs{b}^{-\xi} \leq e^{-n(b)\frac{\xi}{2(\cone + \ctwo)}}$.
Consequently
$\norm{\smash{h -  h_{b}}}_{\L{1}{\base} }  \leq 8 \cly e^{- n(b) \frac{ \xi}{2(\cone + \ctwo)}}$.
\end{proof}

Using Lemma~\ref{lem:approx} we know that in the case $ 2 \cly (1+\abs{b} )\norm{h}_{\L{1}{\base} } > \norm{h}_{\mathbf{BV}}$ then
\[
 \norm{\cL_{b}^{\nn} h}_{\L{1}{\base} }
 \leq 
  \norm{\cL_{b}^{\nn} h_{b}}_{\L{1}{\base} }
  + 
 4 \cly e^{-\nn \frac{\xi}{2(\cone + \ctwo)}}   \norm{h}_{\L{1}{\base}},
\]
since $h_{b} = \sum_{\ell} h_{b} \one_{\ell}$ where $h_{b}$ is constant on each interval $\Hhh_{\ell}$ and that  $\norm{ h }_{\L{1}{\base}} = \norm{ h_{b} }_{\L{1}{\base}} $. 
We now take advantage of the following result. This is the main estimate which takes advantage of the oscillatory cancellation mechanism which is present in this setting.

\begin{prop}\label{prop:mainlem}
There exists $\cm>0$, $\cyes>0$ such that, for all  $\abs{b}\geq b_{0}$ and  $\ell$,
\[
 \norm{\smash{    \cL_{b}^{\nn(b)}\one_{\Hhh_{\ell}}   }}_{\L{1}{\base}}
\leq
\cm e^{-\nn(b) \cyes}
\norm{\one_{\Hhh_{\ell}}}_{\L{1}{\base}}.
\]
\end{prop}
\noindent
 The proof of the above is postponed to Section~\ref{sec:mainestimate}.

Combining Lemma~\ref{lem:approx} and Proposition~\ref{prop:mainlem} we obtain the estimate
\[
\begin{aligned}
\norm{\smash{ \cL_{b}^{\nn(b)} h  }}_{\L{1}{\base} }
 &\leq
\left(  4 \cly e^{-\nn(b)\frac{\xi}{2(\cone + \ctwo)}}
+
\cm e^{-\nn(b) \cyes}
\right)
\norm{h}_{\L{1}{\base}}\\
& \leq
 \cother e^{-\nn(b) \cseven} 
\norm{h}_{\L{1}{\base}}
\end{aligned}
\]
where $\cseven:= \min(\frac{\xi}{2(\cone + \ctwo)},   \cyes)$
and $\cother:= 4 \cly + \cm$.
We now substitute these estimates into \eqref{eq:iterate}.
\[
\begin{aligned}
\norm{\smash{ \cL_{b}^{2\nn(b)} h }}_{(b)}
&\leq
\left(    
2 \cly^{2} \lambda^{-\nn(b)}
+  \cly \cother e^{-\nn(b) \cseven}
\right)
\norm{ h}_{(b)}\\
&\leq
\cly(2\cly + \cother) e^{-\nn(b) \ceight}
\norm{ h}_{(b)},
\end{aligned}
\]
where $\ceight := \min( \ln \lambda, \cseven)$.
To complete the proof of Proposition~\ref{prop:mainuse} we must combine the above estimate with Lemma~\ref{lem:easy}.
We choose $b_{0}>0$ sufficiently large such that
\[
\norm{\smash{ \cL_{b}^{2\nn(b)} h }}_{(b)}
\leq
 e^{-\nn(b) \frac{\ceight}{2}}
\norm{ h}_{(b)}
\]
for all $\abs{b} \geq b_{0}$.
Note that the estimate of Lemma~\ref{lem:easy} cannot be simply iterated since the assumption of the estimate is not invariant. However we can argue as follows: Either the estimate can be interated or the above estimate applies. Consequently we obtain the exponential rate as required and complete the proof of Proposition~\ref{prop:mainuse}.

\section{The Main Estimate}
\label{sec:mainestimate}

This section is devoted to the proof of Proposition~\ref{prop:mainlem} which was stated in Section~\ref{sec:preparing}.
In order to prove this proposition we must estimate $ \norm{\smash{    \cL_{b}^{\nn(b)}\one_{\Omega}   }}_{\L{1}{\base}}$ where $\Omega$ is an interval such that $\abs{b}^{-(1+\xi)} \leq \abs{ \Omega} \leq 2\abs{b}^{-(1+\xi)}$.
Let $\Omega_{0} = \Omega$ and, using the notation of Lemma~\ref{lem:growth},
denote by $ {\{\omega_{j}\}}_{j}$ the connected components of  $\Omega_{n}$.
Let $h_{j} := \left. f^{n(b)} \right|_{\omega_{j}}^{-1}$.
Note that  $\norm{\one_{\Omega}}_{\L{1}{\base}} =\abs{{\Omega} }$. 
We must estimate
\begin{equation}
\label{eq:termtoestimate}
 \norm{\cL_{b}^{n}\one_{\Omega}}_{\L{1}{\base}}
=
\int_{\base}   \left| \vphantom{\sum } \smash{  \sum_{j}  \left(  { J_{n}}   \cdot e^{ib\tau_{n} }  \right)  \circ h_{j}(z)  \cdot \one_{f^{n}\omega_{j}}}(z) \right| \ dz.
\end{equation}
 Introduce a partition of $\base$ into equal sized subintervals ${\{I_{p}\}}_{p}$ such that 
\begin{equation}
 \label{eq:Ipart}
\abs{b}^{-(1-\xi)}
 \leq 
  \abs{I_{p}} 
  \leq 2 \abs{b}^{-(1-\xi)}. 
\end{equation}
For each $p$, fix some $y_{p} \in I_{p}$ as a reference.
To proceed we would like to ensure that the subintervals $f^{n}\omega_{j}$ make full crossings of the intervals $I_{p}$.
For each $p$ let $G_{p}$ denote the set of indexes $j$ such that $f^{n}\omega_{j} \supset I_{p}$.
Let $G_{p}^{\complement}$ denote the complement of $G_{p}$.
The integrals associated to indexes in the set $G_{p}^{\complement}$  are estimated as follows.
  \[
   \sum_{p}
\int_{I_{p}}   \left| \vphantom{\sum } \smash{  \sum_{j \in G_{p}^{\complement}}  \left(  { J_{n}}  \cdot e^{ib\tau_{n} } \right)  \circ h_{j}(z)  \cdot \one_{f^{n}\omega_{j}}}(z) \right| \ dz 
  \leq
   \sum_{p}  \sum_{j \in G_{p}^{\complement}}
  \abs{    \omega_{j} \cap f^{-n}I_{p} }.
\]
That $j \in G_{p}^{\complement}$
implies that one of the end points of $f^{n}\omega_{j}$
is contained within $I_{p}$. 
  Consequently $ \omega_{j} \cap f^{-n}I_{p} $ is contained within the set
  $\{ x\in \Omega_{n} : r_{n}(x) < \epsilon \}$
  where $\epsilon = \abs{I_{p}}  \leq 2 \abs{b}^{-(1-\xi)}$.
  This means that
  \[
     \sum_{p}  \sum_{j \in G_{p}^{\complement}}
  \abs{    \omega_{j} \cap f^{-n}I_{p} }
  \leq
   \cZ_{\epsilon}\Omega_{n}.
  \]
  Applying the estimate of Lemma~\ref{lem:growth} gives a bound of
\[
\begin{aligned}
   \cZ_{\epsilon}\Omega_{n}
   &\leq
   \beta^{-n}\lambda^{n}  \frac{2\epsilon}{\lambda^{n}} + 
     \epsilon \cgrowth \abs{ \Omega  }\\
     &\leq
     8 \abs{b}^{-(1+\xi)}    \left(  e^{-n \ln \beta}e^{n \frac{2\xi}{\cone + \ctwo}} + 2 \cgrowth  e^{-n \frac{1-\xi}{\cone + \ctwo}} \right).
\end{aligned}
\]
Recalling the definitions of  $\xi$ and $\cone$,  note that $ \frac{2\xi}{\cone + \ctwo} < \frac{2\xi}{\cone} = \xi \ln \lambda$ and so 
$  \ln \beta- \frac{2\xi}{\cone + \ctwo} >  \ln \beta-   \xi \ln \lambda > 0$.
Let $\csurplus := \min( \ln \beta- \frac{2\xi}{\cone + \ctwo}, \frac{1-\xi}{\cone + \ctwo}   ) >0$,
$\csurminus:= 8(1+ 2 \cgrowth )$.
  This means that
  \begin{equation}
  \label{eq:surplus}
    \sum_{p}
\int_{I_{p}}   \left| \vphantom{\sum } \smash{  \sum_{j \in G_{p}^{\complement}}  \left(  { J_{n}}  \cdot e^{ib\tau_{n} } \right)  \circ h_{j}(z)  \cdot \one_{f^{n}\omega_{j}}}(z) \right| \ dz 
  \leq
\abs{\Omega} \csurminus e^{-\csurplus n}.
  \end{equation}

Now we may proceed to estimate \eqref{eq:termtoestimate}  summing only over the indexes $j \in G_{p}$.
Since $\abs{\sum_{k} a_{k}}^{2} =   \sum_{jk} a_{j}\overline{a_{k}} $, using also Jensen's inequality, we have 
\begin{equation}\label{eq:jensen}
\begin{aligned}
\sum_{p} \int_{I_{p}}   \left|  \vphantom{\sum } \smash{  \sum_{j\in G_{p}} }
\left( { J_{n}}  \cdot  e^{ib\tau_{n} }\right)
\circ h_{j}(z) \right| \ dz
&= 
\sum_{p} \int_{I_{p}}   \left( \sum_{j,k \in G_{p}} (K_{j,k} \cdot e^{ib \theta_{j,k} })(z) \right)^{\frac{1}{2}} \ dz\\
&\leq 
  \left( \sum_{p}  \sum_{j,k \in G_{p}} \left|  \int_{I_{p}} (K_{j,k} \cdot   e^{ib\theta_{j,k} })(z)    \ dz \right| \right)^{\frac{1}{2}}\\
\end{aligned}
\end{equation}
where $K_{j,k}:= { J_{n}}\circ h_{j} \cdot  { J_{n}}\circ h_{k}$ and we define the following crucial quantity related to the phase difference between different preimages of the same point:
\[
 \theta_{j,k}(x) := \left(  \tau_{n} \circ h_{j  } - \tau_{n} \circ h_{k }  \right)(x).
\]

\begin{lem}
 \label{lem:estimateJp}
 There exists $\cjp >0$ such that 
 $J_{n}'(x) \leq \cjp$ for all $x \in \bT^{1}$, $n\in \bN$.
\end{lem}
\begin{proof}
 Note that
$ J_{n} = \prod_{j=0}^{n-1} \frac{1}{f'}\circ f^{j}  $. 
Consequently
$J_{n}' = \sum_{j=0}^{n-1} \frac{f''}{f'} \circ f^{j} \cdot J_{n-j}\circ f^{j}$.
And so $\abs{J_{n}'} \leq \sup \abs{f''} /(\lambda -1)$ for any $n\in \bN$.
\end{proof}

\begin{lem}\label{lem:taupprime}
There exists $\ctauq>0$, independent of $n \in \bN$, such that 
$
\abs{\smash{\theta_{j,k}''}} \leq \ctauq$.
\end{lem}
\begin{proof}
Suppose that $g: \omega \to \base$ such that
$g \circ f^{n} = \mathbf{id}$.
Let  $g^{(j)} := f^{n-j} \circ g$.
Note that
\[
(\tau_{n}\circ g)' = \sum_{j=0}^{n-1} (\tau' \cdot J_{j})\circ g^{(j)}.
\]
Consequently 
\[
(\tau_{n} \circ g)'' =
 \sum_{j=0}^{n-1} 
 \left(  J_{j}^{2}\cdot \tau'' + \tau' \cdot J_{j}' \cdot J_{j}    \right)\circ g^{(j)}.
\]
By Lemma~\ref{lem:estimateJp} we know that $J_{n}' \leq \cjp$.
Since $\tau$ is $\cC^{2}$ and $J_{n} \leq \lambda^{n}$  the above term is uniformly bounded for any $n\in \bN$.
\end{proof}

Let $g_{j} := f^{n_{1}} \circ h_{j}$.
For each $p$ let $A_{p}$ denote the set of denote the set of pairs $(j,k) \in G_{p} \times G_{p}$ such that $g_{j}(x_{p}) \pitchfork g_{k}(x_{p})$ (this is the case where we see oscillatory cancelations since the two preimages are transversal at iterate $n_{2}$). 
Let $A_{p}^{\complement}$ denote the complement set, i.e., the set of pairs $(j,k)$ such that $g_{j}(x_{p}) \not\pitchfork g_{k}(x_{p})$.

\begin{equation}
\label{eq:thetwopieces}
\begin{aligned}
\sum_{j,k \in G_{p} } \left|  \int_{I_{p}}   (K_{j,k} \cdot   e^{ib\theta_{j,k} })(z)   \ dz \right|
&\leq   \sum_{j}  \sum_{k: (j,k) \in A_{p}} \left|  \int_{I_{p}}   (K_{j,k} \cdot   e^{ib\theta_{j,k} })(z)    \ dz \right|  \\
  & \ \ +   \sum_{j}   \sum_{ k : (j,k) \in A_{p}^{\complement} }  \int_{I_{p}}  K_{j,k}(z)     \ dz.
\end{aligned}
\end{equation}

Before estimating the above it is convenient to give the following distortion estimates.

\begin{lem}
 \label{lem:estK}
$K_{j,k}' \leq 2 \cjp K_{j,k}$.
\end{lem}
\begin{proof}
 Differentiating we obtain
 $ K_{j,k}' = K_{j,k} \left( J_{m}'\circ h_{j} + J_{m}'\circ h_{k}   \right)   $.
By Lemma~\ref{lem:estimateJp} we know that $J_{n}' \leq \cjp$.
\end{proof}

 Recall that  that $\abs{b}^{-(1+\xi)} \leq \abs{ \Omega} \leq 2\abs{b}^{-(1+\xi)}$ and 
 $n_{1} = \lceil \cone \ln \abs{b} \rceil$.
\begin{lem}
\label{lem:otherdistortion}
 For all $y\in \base$,
\[
\sum_{x \in f^{-n_{1}}(y) \cap \Omega} J_{n_{1}}(x) \leq 6 \cly \abs{\Omega}.
\]
\end{lem}
\begin{proof}
First note that
 \[
  \sum_{x \in f^{-n_{1}}(y) \cap \Omega} J_{n_{1}}(x) 
  =
    \sum_{x \in f^{-n_{1}}(y) } J_{n_{1}}(x) \cdot \one_{\Omega}
    = (\cL_{0}^{n_{1}}  \one_{\Omega} )(y),
 \]
 and so it suffices to estimate $\norm{  \cL_{0}^{n_{1}}  \one_{\Omega} }_{\L{\infty}{\base}}$.
 In one dimension $\norm{  \cdot }_{\L{\infty}{\base}} \leq 2 \norm{  \cdot }_{\mathbf{BV}}$.
 Moreover $ \norm{   \one_{\Omega} }_{\L{1}{\base}} = \abs{\Omega}$ and 
 $ \norm{   \one_{\Omega} }_{\mathbf{BV}} = 2$.
 So, with the help of the estimate from Lemma~\ref{lem:LY},
 \[
  \norm{  \cL_{0}^{n_{1}}  \one_{\Omega} }_{\L{\infty}{\base}}
  \leq 2(2 \cly \lambda^{-n_{1}} + \cly \abs{\Omega}).
 \]
 Note that $\lambda^{-n_{1}} = \abs{b}^{-2}$ since $n_{1} \geq \cone  \ln \abs{b}$ and $\cone = 2 / \ln \lambda$.
 Consequently ($\xi \leq \frac{1}{2}$) the above quantity is bounded by $6 \cly \abs{b}^{-(1+\xi)}$.
 \end{proof}
In a similar way to the above
$\sum_{x \in f^{-n}(y)} J_{n}(x) 
    = (\cL_{0}^{n}  \one )(y)$ for all $y\in \bT^{1}$, $n\in \bN$.
    We may again apply the estimate from Lemma~\ref{lem:LY}
and so
\begin{equation}
\label{eq:distortion}
\sum_{x \in f^{-n}(y)} J_{n}(x) \leq \cly.
\end{equation}
 By Lemma~\ref{lem:estK} we know that
$K_{j,k}' \leq 2 \cjp K_{j,k}$.
Hence, by Gronwall's inequality, 
\begin{equation}
 \label{eq:gronwall}
 K_{j,k}(z) \leq e^{2 \cjp \abs{I_{p}}} K_{j,k}(x_{p})
\end{equation}
 for all $z\in I_{p}$.
 Choosing $b_{0}$ large insures that $\abs{I_{p}}$ is small and so
\begin{equation*}
  \int_{I_{p}}  K_{j,k}(z)    \ dz  
  \leq 
2
\abs{{I_{p}}}  K_{j,k}(x_{p}).
\end{equation*}
Now we consider the sum in~\eqref{eq:thetwopieces} corresponding to the non-cancelling pairs (this is the second of the two terms on the right hand side).
Using the above estimates
\[
\sum_{p}  \sum_{j}   \sum_{ k : (j,k) \in A_{p}^{\complement} }  \int_{I_{p}}  K_{j,k}(z)     \ dz
  \leq
     \sum_{j}   \sum_{ k : (j,k) \in A_{p}^{\complement} }   K_{j,k}(x_{p})     \ dz.
\]
Note that
\[
\sum_{j} J_{n} \circ h_{j}(x_{p}) 
\leq
\sum_{y \in f^{-n}(x_{p}) \cap \Omega} J_{n}(y)
\leq
\sum_{z \in f^{-n_{2}}(x_{p}) } J_{n_{2}}(z)
\sum_{y \in f^{-n_{1}}(z) \cap \Omega} J_{n_{1}}(y)
\]
where $n_{1} = \cone \ln \abs{b}$,  $n_{2} = \ctwo \ln \abs{b}$.
Using also the estimate of Lemma~\ref{lem:otherdistortion}
\begin{equation}
\label{eq:toshow}
\sum_{j} J_{n} \circ h_{j}(x_{p}) 
\leq
6 \cly  \abs{\Omega}
\sum_{z \in f^{-n_{2}}(x_{p}) } J_{n_{2}}(z).
\end{equation}
Using the above estimates, together with \eqref{eq:distortion} and \eqref{eq:transdecay},
\begin{equation}
\label{eq:estimate2}
\begin{aligned}
\sum_{p}\sum_{j} \sum_{k: (j,k) \in A_{p}^{\complement}}  \int_{I_{p}} K_{j,k}(z)    \ dz  
& \leq
(6 \cly  )^{2} \abs{\Omega}^{2} 
\cly
\sum_{\substack{z_{2} \in f^{-n_{2}}(x_{p})\\ z_{1} \not\pitchfork z_{2}} } J_{n_{2}}(z_{2})\\
& \leq (6 \cly )^{2}  \abs{\Omega}^{2} \cly \cgam e^{-n_{2}\ctransrate}.
\end{aligned}
\end{equation}

Let us now consider the case where 
$(j,k) \in A_{p}$
and so estimate the remaining term of \eqref{eq:thetwopieces}.
\begin{lem}\label{lem:yep}
Suppose that $(j,k) \in A_{p}$. Then 
\[
\abs{\smash{\theta_{j,k}'(x_{p})}} > \tfrac{1}{2}  \ctaup (J_{n_{2}} \circ g_{j} + J_{n_{2}} \circ g_{k})(x_{p}).
\]
\end{lem}
\begin{proof}
Differentiating, since $\tau_{n}\circ h_{j} = \tau_{n_{1}}\circ h_{j} + \tau_{n_{2}}\circ f^{n_1} \circ h_{j}$, we obtain
\begin{equation*}
 \theta_{j,k}'
 =
 (\tau_{n_{1}}' \cdot J_{n})\circ h_{j} - ( \tau_{n_{1}}' \cdot J_{n}) \circ h_{k}
 + (\tau_{n_{2}}'\cdot J_{n_{2}}) \circ g_{j}   - ( \tau_{n_{2}}' \cdot J_{n_{2}}) \circ g_{k}. 
\end{equation*}
Applying the estimate of Lemma~\ref{lem:tauprime} means that the first two terms can be estimated as
\[
\abs{(\tau_{n_{1}}' \cdot J_{n})\circ h_{j} - ( \tau_{n_{1}}' \cdot J_{n}) \circ h_{k})  }
\leq
 \tfrac{1}{2}  \ctaup (J_{n_{2}} \circ g_{j} + J_{n_{2}} \circ g_{k}).
\]
Using the estimate of Lemma~\ref{lem:transverse} we have that
\[
\abs{(\tau_{n_{2}}'\cdot J_{n_{2}}) \circ g_{j}   - ( \tau_{n_{2}}' \cdot J_{n_{2}}) \circ g_{k} }
>
  \ctaup (J_{n_{2}} \circ g_{j} + J_{n_{2}} \circ g_{k}).\qedhere
\]
\end{proof}

The above lemma says that we have the required transversality at the point $x_{p}$. The following lemma says that the interval $I_{p}$ has been chosen sufficiently small such that this same transversality holds for the entire interval $I_{p}$.
\begin{lem}\label{lem:yepyep}
Suppose that $(j,k) \in A_{p}$. Then $\abs{\smash{   \theta_{j,k}'(y)   }} > \frac{1}{2} \ctaup \Lambda^{-n_{2}}$ for all $y\in I_{p}$. 
\end{lem}
\begin{proof}
By Lemma~\ref{lem:taupprime}  and Lemma~\ref{lem:yep}
we know that
$\abs{\smash{  \theta_{j,k}'(y)   }} 
>
 \ctaup \Lambda^{-n_{2}} 
 - \abs{I_{p}} \ctauq$. 
 To complete the proof it remains to show that $ \abs{I_{p}}  \leq \frac{\ctaup}{2 \ctauq} \Lambda^{-n_{2}}$.
Recall that, by choice of the partition, $\abs{I_{p}} \leq 2 \abs{b}^{-(1-\xi)}$ 
and note that $\Lambda^{-n_{2}} \leq \abs{b}^{-\rho_{2} \ln \Lambda}$.
This means that
$\abs{I_{p}} \leq \Lambda^{n_{2}} \abs{b}^{-(1- \xi - \ctwo \ln\Lambda )}$.
Furthermore, by choice of $\xi$ and $\ctwo$ we have $\ctwo = \frac{\xi}{2\ln\Lambda}$ and $\xi \leq \frac{1}{2}$.
Consequently
$\abs{I_{p}} \leq \Lambda^{n_{2}} \abs{b}^{-\frac{1}{4}}$
and so, again increasing $b_{0}$ if required,
$ \abs{I_{p}}  \leq \frac{\ctaup}{2 \ctauq} \Lambda^{-n_{2}}$
for all $\abs{b}\geq b_{0}$. 
\end{proof}

 The key part of the argument is the following lemma concerning oscillatory integrals.
\begin{lem}\label{lem:vandercorput}
 Suppose $J$ is an interval,  $\theta \in \cC^{2}(J,\bR)$, $K \in \cC^{1}(J,\bC)$, $b\in \bR \setminus \{0\}$ and there exists $\kappa>0$ such that $\inf \abs{\theta'} \geq \kappa$. 
 Then
 \[
  \abs{ \int_{J} K \cdot e^{ib \theta (x) } \ dx } 
  \leq 
   \frac{1}{\abs{b}}  \left(  \tfrac{1}{ \kappa} \sup \abs{K}  + \tfrac{1} { \kappa^2} \sup \abs{K}  \sup \abs{\theta''} \abs{J}  +  \tfrac{1 }{\kappa} \sup \abs{ K'} \abs{J} \right).
 \]
\end{lem}
\begin{proof}[Proof of Lemma~\ref{lem:vandercorput}]
 First change variables, $y=\theta(x)$,   then integrate by parts
 \[
 \begin{aligned}
  \int_{J}  K \cdot  e^{ib \theta (x) } \ dx
  & = \int_{\theta(J)} \frac{K}{\theta'}\circ \theta^{-1}(y) e^{ib y } \ dy\\
  & = 
  \frac{i}{b}  \left[  \frac{K}{\theta'}\circ \theta^{-1}(y)   e^{ib y }  \right]_{\theta(J)} \\
 & \ \ \ \ + \frac{i}{b} \int_{\theta(J)} \left(  \frac{K  \theta''}{(\theta')^2 \cdot \theta'} +  \frac{K'}{(\theta')^2} \right) \circ \theta^{-1}(y) e^{ib y } \ dy.
    \end{aligned}
 \]
Changing variables again, we obtain
 \[
  \int_{J} K \cdot e^{ib \theta (x) } \ dx
= 
  \frac{i}{b}   \left[  \frac{K}{\theta'}   e^{ib \theta }  \right]_{J}    
  + \frac{i}{b} \int_{J} \left( \frac{ K\theta''}{(\theta')^2}+ \frac{K'}{\theta'}  \right) (x) e^{ib \theta(x) } \ dx.
 \]
 The required estimate follows immediately.
\end{proof}

 In preparation of apply the above lemma, note that \eqref{eq:gronwall} implies
 $\sup_{I_{p}} K_{j,k} \leq 2 K_{j,k}(x_{p})$
 and similarly 
  $\sup_{I_{p}} K_{j,k}' \leq 4 \cjp K_{j,k}(x_{p})$.
By Lemma~\ref{lem:yepyep} we know that 
$\abs{\smash{   \theta_{j,k}'   }} > \frac{1}{2} \ctaup \Lambda^{-n_{2}}$.
By Lemma~\ref{lem:taupprime} we know that 
$\abs{\smash{\theta_{jk}''}} \leq \ctauq$.
Using  these estimates with Lemma~\ref{lem:vandercorput} we obtain
\[
  \left|  \int_{I_{p}} ( K_{j,k} \cdot  e^{ib\theta_{j,k} })(z)    \ dz \right| 
 \leq 
  \frac{1}{\abs{b}} K_{j,k}(x_{p}) 
\cend \Lambda^{ n_{2}}
\]
where $\cend := 8\cjp (1+\ctauq + 2 \cjp)\ctaup^{-2}$.
This means that,  for the first sum in \eqref{eq:thetwopieces}, we obtain, using \eqref{eq:distortion} and Lemma~\ref{lem:otherdistortion}, the estimate
\begin{equation*}
\sum_{p} \sum_{\substack{j,k \\ (j,k) \in A_{p}}} \left|  \int_{I_{p}} ( K_{j,k} \cdot  e^{ib\theta_{j,k} })(z)    \ dz \right| 
 \leq 
\frac{1}{\abs{b}} \abs{b}^{1-\xi} \cend \Lambda^{n_2} \cly^{2} (6 \cly )^{2} \abs{\Omega}^{2}
\end{equation*}
(the term $\abs{b}^{1-\xi}$ comes from the sum over $p$).
Since $\abs{b} \geq e^{n_{2} \frac{2 \xi}{3  \cone}}$ (increasing $b_{0}$ again if required),
\[
 \frac{1}{\abs{b}} \abs{b}^{1-\xi}  \Lambda^{n_2}
 =
 \frac{1}{\abs{b}^{\xi}}  \Lambda^{n_2}
 \leq e^{-n_{2}(\frac{2\xi}{3\ctwo} - \ln \Lambda )}.
\]
Let $\cyonem := \frac{2\xi}{3\ctwo} - \ln \Lambda > 0$. This means that
\begin{equation}
 \label{eq:estimate1}
\sum_{p} \sum_{\substack{j,k \\ (j,k) \in A_{p}}} \left|  \int_{I_{p}} ( K_{j,k} \cdot  e^{ib\theta_{j,k} })(z)    \ dz \right| 
\leq 
\cend \cly^{2} (6 \cly )^{2} \abs{\Omega}^{2}    e^{-n_{2}\cyonem}.
\end{equation}
In order to estimate the final term in~\eqref{eq:jensen} we use~\eqref{eq:thetwopieces}, sum the estimates \eqref{eq:estimate2}  and \eqref{eq:estimate1},   to obtain 
\[
\sum_{p}  \sum_{j,k \in G_{p}} \left|  \int_{I_{p}} (K_{j,k} \cdot   e^{ib\theta_{j,k} })(z)    \ dz \right| 
 \leq 
\abs{\Omega}^{2}
e^{-n \calmost}
\left( (6 \cly )^{2}   \cly \cgam  +  \cend \cly^{2} (6 \cly )^{2} \right),
 \]
 where $\calmost := \frac{\cone}{\cone + \ctwo} \min(\ctransrate , \cyonem) >0$.
Let $\cm:=6 \cly  \left(   \cly \cgam  +  \cend \cly^{2}\right)^{\frac{1}{2}}$. Taking the square root 
of the above
 \[
\begin{aligned}
\sum_{p} \int_{I_{p}}   \left|  \vphantom{\sum } \smash{  \sum_{j\in G_{p}} }
\left( { J_{n}} \cdot  e^{ib\tau_{m} }\right)\circ h_{j}(z) \right| \ dz
 &\leq \abs{\Omega} e^{-n \frac{\calmost}{2}} \cm
\end{aligned}
\]
Including also the estimate \eqref{eq:surplus} we have shown that (let $\cyes:= \min (\csurplus,\frac{\calmost}{2} )>0$)
\[
\int_{\base}   \left| \vphantom{\sum } \smash{  \sum_{j}  \left(  { J_{n}}  \cdot e^{ib\tau_{n} }  \right)  \circ h_{j}(z)  \cdot \one_{f^{n}\omega_{j}}}(z) \right| \ dz 
\leq
(\cm + \csurminus) \abs{\Omega} e^{-n(b) \cyes}.
\]
This completes the proof of Proposition~\ref{prop:mainlem}.

\section{Rate of Mixing}
\label{sec:rateofmixing}
Here we use the estimates of Proposition~\ref{prop:mainuse} concerning the twisted transfer operators in order to estimate the rate of mixing.
Let $g,h : \bT^{2} \to \bC$ be two observables. We assume, without loss of generality, that $g$ is mean zero. 
Denote by $\hat{g}_{b}$ and $\hat{h}_{b}$ their Fourier components (in the fibre coordinate), i.e.,
\[
 g(x,u) = \sum_{b\in \bZ} \hat{g}_{b}(x) e^{-ibu},
\]
and similarily for $h(x,u)$. 
Using the regularity of the observables (in particular the smoothness in the fibre direction) we have that $\norm{ \hat{g}_{b} }_{\mathbf{BV}} \leq (1+\abs{b})^{-1} \norm{g}_{\cC^{1}}$ (and for $h$ similarily). 
By simple manipulations we obtain the formula
\[
 \int_{\bT^{2}} \left( g \cdot h\circ F\right)(x,u) \ dx \ du
 =
 \sum_{b\in \bZ} \int_{\bT^{1}} \cL_{b}^{n} \hat{g}_{b}(x) \cdot \hat{h}_{b}(x) \ dx.
\]
We separate the sum into a finite number of terms where $\abs{b} \leq b_{0}$ and the infinite sum of the remaining terms. For the finite number of terms it suffices to use the quasi-compactness in a standard way using that the base map is mixing.

Recall that in Proposition~\ref{prop:mainuse} we obtained the estimate 
$  \norm{\smash{\cL_{b}^{n(b)}} }_{(b)}
  \leq
  e^{-n(b) \cten }$
where $  \cfour \ln \abs{b} \leq n(b) \leq  \cfour \ln \abs{b}  +2$.
We may assume that $\cten >0$ is sufficiently small that $\cfour \cten <1$.
Consequently the above estimate implies that there exists $\alpha\in (0,1)$ such that
$  \norm{\smash{\cL_{b}^{n}} }_{(b)}
  \leq
  \abs{b}^\alpha e^{-n \cten }$
  for all $n\in \bN$, $\abs{b} \geq b_{0}$.
  Note that
   \[
   \abs{\int_{\bT^{1}} \cL_{b}^{n} \hat{g}_{b}(x) \cdot \hat{h}_{b}(x) \ dx}
   \leq  2 \norm{\smash{\cL_{b}^{n}}   }_{(b)}    \norm{\smash{\hat{g}_{b} }  }_{\mathbf{BV}} \norm{\smash{\hat{h}_{b} }  }_{\mathbf{BV}}  
   \]
    since the $\mathbf{BV}$ norm dominates the $\mathbf{L^{\infty}}$ norm.
It remains to observe that
\[
  \sum_{\abs{b} \geq b_{0}} 
  \norm{\smash{\cL_{b}^{n}} }_{(b)}
  \norm{\smash{ \hat{g}_{b} }}_{\mathbf{BV}}
  \norm{\smash{ \hat{h}_{b}} }_{\mathbf{BV}}
  \leq  \sum_{\abs{b} \geq b_{0}}   \abs{b}^{-(2-\alpha)} e^{-n \cten }
  \norm{g}_{\cC^{1}}\norm{h}_{\cC^{1}}.
\]
Crucially $(2-\alpha) >1$ and so this is summable. 
This proves exponential mixing for $\cC^{1}$ observables which, by the usual argument~\cite[footnote~2]{oli1401}, implies exponential mixing for H\"older observables.



\vspace{10pt}

\end{document}